\documentclass[12pt,twoside]{amsart}
\usepackage{amsmath}
\usepackage{amsthm}
\usepackage{amsfonts}
\usepackage{amssymb}
\usepackage{latexsym}
\usepackage{mathrsfs}
\usepackage{amsmath}
\usepackage{amsthm}
\usepackage{amsfonts}
\usepackage{amssymb}
\usepackage{latexsym}
\usepackage{geometry}
\usepackage{dsfont}
\usepackage[dvips]{graphicx}
\usepackage{color}
\usepackage[all]{xy}

\date{}
\allowdisplaybreaks[4] \footskip=15pt
\renewcommand{\uppercasenonmath}[1]{}

\usepackage{graphicx,amssymb}
\usepackage[all]{xy}
\usepackage{amsmath}

\allowdisplaybreaks
\usepackage{amsthm}
\usepackage{color}

\theoremstyle{plain}
\newtheorem{theorem}{Theorem}[section]
\newtheorem{proposition}[theorem]{Proposition}
\newtheorem{problem}[theorem]{Problem}

\newtheorem{example}[theorem]{Example}

\newtheorem{definition}[theorem]{Definition}

\theoremstyle{definition}

\theoremstyle{remark}
\newtheorem{remark}[theorem]{Remark}

\def\m{\frak m}

\def\cf{{\rm cf}}

\begin{document}
\begin{center}
{\large  \bf On the existence of parameterized noetherian rings}

\vspace{0.5cm} Xiaolei Zhang

\end{center}

\bigskip
\centerline { \bf  Abstract}
\bigskip
\leftskip10truemm \rightskip10truemm \noindent

  A ring $R$ is called left strictly $(<\aleph_{\alpha})$-noetherian if $\aleph_{\alpha}$ is the minimum cardinal such that every  ideal of $R$ is $(<\aleph_{\alpha})$-generated.   In this note, we show that for every singular (resp., regular) cardinal $\aleph_{\alpha}$, there is a valuation domain $D$, which is strictly $(<\aleph_{\alpha})$-noetherian (resp., strictly $(<\aleph_{\alpha}^+)$-noetherian), positively answering a problem proposed in \cite{Marcos25} under some set theory assumption.
\vbox to 0.3cm{}\\
{\it Key Words:}    strictly $(<\aleph_{\alpha})$-noetherian ring, valuation domain, weakly inaccessible cardinal.\\
{\it 2020 Mathematics Subject Classification:}  16P40.

\leftskip0truemm \rightskip0truemm
\bigskip

\section{Introduction}
Throughout this note, all rings are associative with unity, and  all ideals and modules  are left. Let $\{r_i\mid i\in\Lambda\}$ be a set of elements of $R$, we denote by $\langle \{ r_i\mid i\in\Lambda\}\rangle$ be the ideal generated by these elements. 

Recall that a ring $R$ is said to be  left \emph{noetherian} if every  ideal is finitely generated. It is interesting to consider the minimum  cardinal $\gamma(R)$ such that every  ideal of $R$ is $(<\gamma(R))$-generated. So noetherian rings have exactly $\gamma(R)\leq \aleph_0$. Trivially, PIRs  have exactly $\gamma(R)=2$, and  Dedekind domains which is not PIDs have  $\gamma(R)=3$. Furthermore, Matson \cite{Matson09} give an example of a ring $R$ with $\gamma(R)=n$ for each $n\geq 2$.

Recently, Marcos \cite{Marcos25}  parametrized   rings by the cardinality of generators of all ideals. Actually, he introduced the notion of $(<\aleph_{\alpha})$-noetherian rings for any cardinal $\aleph_{\alpha}$.  A ring $R$ is left $(<\kappa)$-noetherian if every  ideal of $R$ is $(<\kappa)$-generated, that is, $\gamma(R)\leq \kappa$.
When $\aleph_{\alpha}$ is regular, Marcos \cite{Marcos25} characterized $(<\aleph_{\alpha})$-noetherian rings in terms of limit models and $\aleph_{\alpha}$-injective modules.

We say  a ring $R$ is left strictly $(<\kappa)$-noetherian if $\gamma(R)=\kappa$. The existence of strictly $(<\aleph_{\alpha})$-noetherian rings is key to their studies. When $\aleph_{\alpha+1}$ is a successor cardinal, Marcos \cite{Marcos25} give an example of strictly $(<\aleph_{\alpha+1})$-noetherian ring by $\aleph_{\alpha}$-variables polynomial extensions of the field $\mathbb{Q}$. Subsequently, he \cite{Marcos25} proposed a problem on the existence of
strictly $(<\aleph_{\alpha})$-noetherian rings when $\aleph_{\alpha}$ is a limit cardinal:\\
\textbf{Problem 2.6.} \cite[Problem 2.9]{Marcos25}   Show that for every $($some$)$ limit ordinal $\alpha$ there is a ring which is left $(<\aleph_\alpha)$-noetherian, but not left $(<\aleph_\beta)$-noetherian for every $\beta < \alpha$.

The main motivation of this note is to investigate this problem. Actually, we study the cardinal of generators of ideals of a valuation domain $D$ with valuation group  $G=\bigoplus_{i\in\aleph_\alpha}\mathbb{Q}$. We obtain that if $\aleph_\alpha$ is a regular cardinal, then $D$ is a strictly $(<\aleph_\alpha^{+})$-noetherian domain; if $\aleph_\alpha$ is a singular cardinal, then $D$ is a strictly  $(<\aleph_\alpha)$-noetherian domain (see Theorem \ref{main}). Hence, we positively answer the Problem 2.6 under the set theory assumption: the weakly inaccessible cardinals do not exist.

\section{main result}
We first recall some basic knowledge on set theory (see \cite{J03} for more details). 
Let  $\kappa$ (resp., $\alpha$) be a cardinal (resp., an ordinal). We denote by $\kappa^+$ (resp., $\alpha^+$) the least cardinal (resp., an ordinal) greater than $\kappa$. Let $\Lambda$ be a set, we denote by $|\Lambda|$ its cardinal.

Let $\alpha$ be an ordinal, and $u\in\alpha$, then $\{x\in \alpha\mid x<u\}$ is called an \emph{initial segment} of $\alpha$. Let $\alpha$ be an ordinal. If $\alpha=\beta^+$
for some ordinal $\beta$ (which is denoted by $\alpha^-$), then $\alpha$ is called a \emph{successor ordinal}, otherwise $\alpha$ is called a \emph{limit ordinal}. A cardinal $\aleph_{\alpha}$ is called a \emph{successor cardinal}  provided that  $\alpha$ is  a successor ordinal, otherwise $\aleph_\alpha$ is called a \emph{limit cardinal}. 

Let $\alpha>0$ be a limit ordinal. the \emph{cofinality} of $\alpha$, denoted by $\cf(\alpha)$, is the least limit ordinal $\beta$ such that there is an increasing sequence $\{\alpha_\xi\mid\xi<\beta\}$ with $\lim\limits_{\xi\rightarrow\beta}\alpha_{\xi}=\alpha.$ 
An infinite  cardinal $\kappa$ is said to be \emph{regular} (resp., \emph{singular}) if $\cf(\kappa)=\kappa$ (resp., $\cf(\kappa)<\kappa$). An uncountable cardinal $\kappa$ is said to be \emph{weakly inaccessible} if it is a limit cardinal and is regular. It is well-known that the existence of weakly inaccessible cardinals is consistent with ZFC.

\begin{definition} \cite[Definition 2.1]{Marcos25} 
	Let $\kappa$ be an infinite cardinal. An ideal $I$ of a ring $R$ is $(<\kappa)$-generated if there is $\mu < \kappa$ and $\{a_i\mid i < \mu\} \subseteq I$ such that $I = \langle \{a_i\mid i < \mu\} \rangle$. $I$ is strictly $\kappa$-generated if it is $(<\kappa^+)$-generated but not $(<\kappa)$-generated. 
\end{definition}

The following statement shows that for every for every regular cardinal $\mu <\kappa$, every strictly $\kappa$-generated ideal contain a strictly $\mu$-generated sub-ideal. 

\begin{proposition} \cite[Proposition 2.2]{Marcos25}  Let $R$ be a ring, $I$ an ideal, and $\kappa$ an infinite cardinal. If $I$ is a strictly $\kappa$-generated ideal, then for every regular cardinal $\mu < \kappa$ there is $J$ a strictly $\mu$-generated ideal such that $J \subseteq I$.
\end{proposition}

Given a ring $R$. We denote by $\gamma(R)$ be the minimum infinite cardinal such that every  ideal of $R$ is $(<\gamma(R))$-generated.

The author in \cite{Marcos25} parametrized noetherian rings using the cardinality of  generators of all  ideals.

\begin{definition} \cite[Definition 2.5]{Marcos25} Let $\kappa$ be an infinite cardinal. A ring $R$ is left $(<\kappa)$-noetherian if every  ideal of $R$ is $(<\kappa)$-generated, i.e., $\gamma(R) \leq \kappa$. 
 \end{definition}

Trivially, a left $(<\aleph_0)$-noetherian ring is precisely a left noetherian ring. For convenience, we introduce the notion of strictly $(<\kappa)$-noetherian rings.
\begin{definition} Let $\kappa$ be an infinite cardinal. A ring $R$ is called left strictly $(<\kappa)$-noetherian if $\gamma(R) = \kappa$, that is, every  ideal of $R$ is $(<\kappa)$-generated, and 
\begin{enumerate}
	\item if $\kappa$ is a limit cardinal, then 	for every cardinal $\mu<\kappa$ there is an  ideal of $R$ which  is not $(<\mu^+)$-generated;
	\item if $\kappa$ is a successor cardinal, then there is a strictly $\kappa^-$-generated ideal of $R$. 
\end{enumerate}	
\end{definition}

Note that   $\mathbb{F}[x,y]$ is an example of strictly $(<\aleph_0)$-noetherian ring (see \cite[Example 2.3]{Matson09}). For every successor cardinal $\alpha+1$, Marcos \cite{Marcos25} provided an example of  strictly $(<\aleph_{\alpha+1})$-noetherian ring.

\begin{example} \cite[Example 2.6]{Marcos25}  Let $\alpha$ be an ordinal. Let $R = \mathbb{Q}[x_i\mid i < \aleph_\alpha]$, the polynomial ring over $\mathbb{Q}$ with commuting variables $\{x_i\mid i < \aleph_\alpha\}$. $R$ is $(<\aleph_{\alpha+1})$-noetherian as $|R| = \aleph_\alpha$ and $R$ is not $(<\aleph_\alpha)$-noetherian as $\langle\{ x_i\mid i < \aleph_\alpha\} \rangle$ is not $(<\aleph_\alpha)$-generated.
\end{example} 

Subsequently, he \cite{Marcos25} asked the following problem:
\begin{problem}\label{prob} \cite[Problem 2.9]{Marcos25}   Show that for every $($some$)$ limit ordinal $\alpha$ there is a ring which is left $(<\aleph_\alpha)$-noetherian, but not left $(<\aleph_\beta)$-noetherian for every $\beta < \alpha$, that is, a  left strictly  $(<\aleph_\alpha)$-noetherian ring.
\end{problem}

The main result of this paper is to resolve this problem when $\aleph_{\alpha}$ is a singular cardinal. To move on, we recall some basic notions on valuation domains (see \cite[Chapter II]{fs01}  for more details).

Let \((G,+, \leq)\) be a totally ordered group and $Q$ a field. A map $v: Q \rightarrow G \cup \{\infty\}$, where $\infty$ is a symbol not in $G$ satisfying that $x < \infty$ and $x + \infty = \infty$ for all $x \in G$, is called a \emph{valuation map} if the following conditions hold:

\begin{enumerate}
	\item $v(0) = \infty$,
	\item  $v(xy) = v(x)+v(y)\text{ for all } x, y \in Q.$, and
	\item $v(x+y) \geq \min\{v(x), v(y)\} \text{ for all } x, y \in Q.$
\end{enumerate}

Let $D$ be an integral domain with its quotient field $Q$. Then $D$ is called a \emph{valuation domain} if either $x|y$ or $y|x$ for any $x,y\in D$. It is well-known that an integral domain $D$ is a valuation domain if and only if there is a valuation map $v: Q \rightarrow G \cup \{\infty\}$ such that $D=\{r\in Q\mid v(r)\geq0\}.$ In this case, $G$ is called the \emph{valuation group} of $D$. Note that $\m=\{r\in D\mid v(r)>0\}$ is the maximal ideal of $D$. So if $v(x)=v(y)$, then $x=uy$ for some unit $u$ of $D$. A well-known result introduced by Krull states that every total abelian group can be seen as a valuation group of a valuation domain $D$ (see \cite[Chapter II, Theorem 3.8]{fs01}).

Let $\alpha$ be an ordinal. Let $D$ be a valuation domain, where its valuation group  $G=\bigoplus_{i\in\aleph_\alpha}\mathbb{Q}$ the direct sum of $\aleph_\alpha$-copies of rational numbers $\mathbb{Q}$. It becomes a total ordered abelian group if one defines $r=(\cdots,r_i,\cdots)>0$ whenever for the first element $j$ in $supp(r)$, one  haves $r_j>0$ in $\mathbb{Q}$.

\begin{theorem}\label{main} Let $D$ be the valuation domain as above. Then the following statements hold.
\begin{enumerate}
\item  If $\aleph_\alpha$ is a regular cardinal, then $D$ is a strictly $(<\aleph_\alpha^{+})$-noetherian domain.
\item  If $\aleph_\alpha$ is a singular cardinal, then $D$ is a strictly  $(<\aleph_\alpha)$-noetherian domain.		
\end{enumerate} 
\end{theorem}	

\begin{proof}
(1) Suppose $\aleph_\alpha$ is a regular cardinal, that is, $\cf(\aleph_\alpha)=\aleph_\alpha$. Let $\m$ be the maximal ideal of $D$. Then $\m$ is generated by $\{x_i\in D\mid v(x_i)=e_i, i\in\aleph_\alpha\},$ where $e_i$ is an $\aleph_\alpha$-sequence with the $i$-th component $1$ and others $0$. So $\m$ is $(<\aleph_\alpha^{+})$-generated. We claim that $\m$ is not $(<\aleph_\alpha)$-generated. Indeed, otherwise $\m$ can be generated by $\{m_j\mid j\in \aleph_\beta\}$  elements with $\beta<\alpha$. Since $\cf(\aleph_\alpha)=\aleph_\alpha$, there is $i_0\in\aleph_\alpha$ such that $e_{i_0}<v(m_j)$ for any $j\in\aleph_\beta$. Assume that $v(m)=e_{i_0}$. Then each $m_j$  is divided by $m$, but $m$ can not divided by $m_j$ for any $j\in\aleph_\beta$. Hence $m\not\in \m$, which is a contradiction.

Next, we will show every ideal $I$ of $D$ is $(<\aleph_\alpha^{+})$-generated. Indeed, suppose $I=\langle \{ m_i\mid i\in\Gamma \}\rangle$. Let $\Gamma'$ be a subset of $\Gamma$ satisfying that $\{v(m_i)\mid i\in\Gamma\}=\{v(m_i)\mid i\in\Gamma'\}$ and $v(m_i)\not=v(m_j)$ for every elements $i\not= j$ in $\Gamma'$. It is easy to verify that $I=\langle\{ m_i\mid i\in\Gamma'\}\rangle$. Since the cardinal of $\Gamma'$ is not greater that of $G=\bigoplus_{i\in\aleph_\alpha}\mathbb{Q}$. As the latter is equal to $\aleph_\alpha$, we have $|\Gamma'|\leq \aleph_\alpha$. Hence  $I$  is $(<\aleph_\alpha^{+})$-generated.

Consequently, $D$ is a $(<\aleph_\alpha^{+})$-noetherian domain but not $(<\aleph_\alpha)$-noetherian, that is, $D$ is  a strictly $(<\aleph_\alpha^{+})$-noetherian domain.

(2) Now suppose $\aleph_\alpha$ is a singular cardinal.
Let $I$ be an ideal of $D$. As the proof of the above, we can  show that  $I$ of $D$ is $(<\aleph_\alpha^{+})$-generated. Next, we will show that $I$ is actually $(<\aleph_\alpha)$-generated. Indeed, assume that  $I=\langle \{r_i\mid i\in \aleph_\alpha\}\rangle$ with the property that $r_i|r_j$ when $i>j$ in $\aleph_\alpha$. Since $\aleph_\alpha$ is singular, $\cf(\aleph_\alpha)<\aleph_\alpha$. Then there is a subset $\Gamma$ of $\aleph_\alpha$ with  $|\Gamma|=\cf(\aleph_\alpha)$ such that $I=\langle \{ r_i\mid i\in \Gamma\}\rangle$.  Consequently, $I$ is actually $(<\aleph_\alpha)$-generated.

Next, we will show for any $\beta<\alpha$, there is an ideal $I$ of $D$ such that $I$ is not $(<\aleph_\beta)$-generated. 	Since  $\aleph_\alpha$ is a singular cardinal, $\alpha$ is a limit ordinal. Hence $\beta^{+}< \alpha$. Since $\aleph_{\beta^+}<\aleph_{\alpha}$, there is an initial segment $\Lambda$ of $\aleph_{\alpha}$ with $|\Lambda|=\aleph_{\beta^+}$. Let $I$ be an ideal of $D$ generated by $\{r\in D\mid v(r)=e_i,i\in \Lambda\}$.
Then $I$ can be generated by $\aleph_{\beta^+}$ elements. We claim that  $I$ is not $(<\aleph_{\beta^+})$-generated. Indeed, assume that $I$ is generated by $\aleph$ elements $\{r_j\mid j\in \aleph\}$ with $\aleph<\aleph_{\beta^+}$. So $\aleph$ can be seen as a proper well-ordering subset of $\Lambda$. Since $\aleph_{\beta^+}$ is regular, there is $i_0\in\Lambda-\aleph$ such that $e_{i_0}<v(r_j)$ for any $j\in\aleph$ as in $(1)$. Assume that $v(r)=e_{i_0}$. Then each $r_j$  is divided by $r$, but $r$ can not divided by $r_j$ for any $j\in\aleph$. Hence $r\not\in I$, which is a contradiction.

Consequently, $D$ is a $(<\aleph_\alpha)$-noetherian domain but not $(<\aleph_\beta)$-noetherian for any $\beta<\alpha$, that is, $D$ is  a strictly $(<\aleph_\alpha)$-noetherian domain.		
\end{proof}

\begin{remark} Suppose $\aleph_\alpha$ is a singular cardinal, then $\alpha$ is a limit ordinal and $D$ is the required strictly $(<\aleph_\alpha)$-noetherian ring in Problem \ref{prob}. Since the non-existence of weakly inaccessible cardinals is consistent with ZFC, the existence of strictly $(<\aleph_\alpha)$-noetherian ring for all limit ordinal $\alpha$ is consistent with ZFC.
\end{remark}


\begin{thebibliography}{99}

 \bibitem{fs01}   L. Fuchs and L. Salce, Modules over Non-noetherian Domains. American Mathematical Society, 2001.
 

\bibitem{J03} J. Thomas, Set theory. Springer Monographs in Mathematics. Springer-Verlag, Berlin,
 2003. The third millennium edition, revised and expanded.
 \bibitem{Matson09} A. Matson, 
 Rings of finite rank and finitely generated ideals.
 J. Commut. Algebra 1, No. 3, 537-546 (2009).
 
 \bibitem{Marcos25} M.  Mazari-Armida,  On limit models and parametrized noetherian rings. J. Algebra 669, 58-74 (2025).

\end{thebibliography}
\end{document}